\documentclass[11pt]{article}

\usepackage[T1]{fontenc}
\usepackage{lmodern}
\usepackage[margin=1.02in]{geometry}
\usepackage{microtype}
\usepackage{amsmath,amssymb,amsthm,mathtools}
\usepackage{booktabs}
\usepackage{array}
\usepackage{graphicx}
\usepackage[hidelinks]{hyperref}
\hypersetup{
  pdftitle={Twisted Bernoulli Zeros in Quasi-Linear Time: Distribution, Depth,
            and Explicit Hilbert Class Components},
  pdfauthor={Peter Chocian},
  pdfkeywords={cyclotomic fields, generalized Bernoulli numbers, circular units,
               Hilbert class fields, irregular primes, chirp transform}
}

\allowdisplaybreaks
\numberwithin{equation}{section}

\newtheorem{theorem}{Theorem}[section]

\newtheorem{lemma}[theorem]{Lemma}
\theoremstyle{definition}

\theoremstyle{remark}
\newtheorem{remark}[theorem]{Remark}

\newcommand{\Q}{\mathbf Q}
\newcommand{\Z}{\mathbf Z}
\newcommand{\F}{\mathbf F}

\newcommand{\Cl}{\operatorname{Cl}}
\newcommand{\length}{\operatorname{length}}

\newcommand{\nZerosFthree}{4,812}
\newcommand{\nZerosFfive}{4,654}

\newcommand{\nZerosTotal}{27,508}
\newcommand{\chiSqFthree}{0.05}
\newcommand{\chiSqFfive}{3.93}
\newcommand{\ksD}{0.0061}
\newcommand{\ksCrit}{0.0140}
\newcommand{\ksN}{9,465}
\newcommand{\conjBoth}{699}
\newcommand{\conjBothExp}{739.8}
\newcommand{\conjCorr}{0.013}
\newcommand{\conjPrimes}{4,782}
\newcommand{\rateReg}{0.497}
\newcommand{\rateIrr}{0.509}
\newcommand{\chiSqInd}{1.31}
\newcommand{\fracReg}{59.9}

\title{Twisted Bernoulli Zeros in Quasi-Linear Time:\\
Distribution, Depth, and Explicit Hilbert Class Components}
\author{Peter Chocian}
\date{9 August 2026}

\begin{document}
\maketitle

\begin{abstract}
Let $\chi$ be an odd primitive Dirichlet character of conductor $f$ and
order $d$, and let $p\nmid f$ be a prime with $d\mid p-1$.  Under the
additional semisimplicity hypothesis $p\nmid\varphi(f)$, the divided
generalized Bernoulli values
$b_{\chi,j}=\overline{fB_{1,\chi\omega^{-j}}}\in\F_p$ control, through
the characterwise abelian Main Conjecture, the odd isotypic components
of the $p$-class group of $\Q(\zeta_{fp})$.  In the range proven in the
companion papers \cite{TwistedF3,FourierF5}, their zeros produce
explicit circular-unit Kummer generators of the corresponding Hilbert
class field components.

We first show that a single residue-class weight formula reduces the
computation of $b_{\chi,j}$ at arbitrary conductor to a weighted power
sum over $\F_p$, and that a Bluestein chirp factorization then
evaluates the complete spectrum $(b_{\chi,j})_{2\le j\le p-3}$ in
$O(M(p))=\widetilde O(p)$ field operations for fixed conductor, where
$M(p)$ denotes the cost of multiplying degree-$p$ polynomials over
$\F_p$.  This is a direct, finite-field-native alternative to the
generating-function/power-series route, of the same quasi-linear
order; its practical simplicity makes exhaustive fixed-conductor
surveys convenient.

We then report such a survey: all even indices for
$f\in\{3,5\}$ with $p<10^5$, and for every odd primitive character of
conductor $f\le20$ with $p<2\cdot10^4$; in total $\nZerosTotal$ zero
lines over $55{,}121$ character--prime pairs, each zero line certified
by two independent code paths and carrying an exact second-order
digit.  The counts are consistent with the Poisson model
($\chi^2$ statistics $\chiSqFthree$ and $\chiSqFfive$ at conductors
$3$ and $5$), the divided digits show no statistically significant
departure from uniformity, conjugate quartic characters never vanish
at a common index in the tested range, and no statistically significant
association between twisted degeneracy and classical irregularity is
detected; these observations complement, at fixed conductor and much
larger $p$, the $\lambda$-invariant distribution studies of
Delbourgo--Knospe \cite{DelbourgoKnospe} and Knospe \cite{Knospe},
for which a $b_{\chi,j}$ zero corresponds to a $\chi$-irregular branch
(positive $\lambda$-invariant).  Eight
zero lines in our range are non-simple, including one of depth three
at $(f,p,j)=(19,37,16)$; the associated
$\overline\chi\omega^{16}$-component of
$\Cl(\Q(\zeta_{703}))\otimes\Z_{37}$ then has order $37^3=50{,}653$.
The paper's principal object is not the spectrum itself but its
conversion into \emph{explicit certified generators}: the
conductor-three catalogue theorem of \cite{TwistedF3} is extended
from $p<500$ to $p<10^5$, giving $2{,}441$ zero lines with
$p\equiv1\pmod 6$, all simple, each carrying a fresh split-prime
Artin certificate, so that each projected circular unit generates the
complete order-$p$ component of the Hilbert class field.  The largest,
at $p=99{,}991$, is an explicitly generated component of order
$99{,}991$ in the degree-$199{,}980$ field $\Q(\zeta_{299973})$,
certified in under a second, in a range where general-purpose
unconditional class group computation is not available.  A
deterministic integer-arithmetic program reproduces every table.
\end{abstract}



\section{Introduction}

Let $\chi$ be an odd primitive Dirichlet character of conductor $f>1$
and exact order $d$, let $p\nmid f$ be an odd prime with $d\mid p-1$,
and fix an embedding of the values of $\chi$ into $\F_p$.  Let
$\omega$ be the Teichm\"uller character modulo $p$.  For even $j$,
$2\le j\le p-3$, the divided twisted Bernoulli value
\begin{equation}
\label{eq:b-def}
 b_{\chi,j}
 =\overline{\;\frac1p
 \sum_{\substack{1\leq a<fp\\(a,fp)=1}}
 a\,\chi(a)\widetilde a^{-j}\;}
 =\overline{fB_{1,\chi\omega^{-j}}}\in\F_p ,
\end{equation}
where $\widetilde a$ denotes the Teichm\"uller lift of $a\bmod p$,
occupies the same position for the field $K_p=\Q(\zeta_{fp})$ that the
Bernoulli numerators $B_k$ occupy for $\Q(\zeta_p)$ \cite{Washington}: by the
finite-level odd-character consequence of the abelian Main Conjecture
\cite{MazurWiles,Greither},
\begin{equation}
\label{eq:MC}
 \length_{\Z_p}\bigl(\Cl(K_p)\otimes\Z_p\bigr)_\psi
 =v_p\bigl(B_{1,\psi^{-1}}\bigr),
 \qquad
 \psi=\overline\chi\,\omega^{j},
\end{equation}
valid whenever $\psi$ is odd, primitive of conductor $fp$, distinct
from $\omega$, and $p\nmid[K_p:\Q]$.  A zero $b_{\chi,j}=0$ is thus a
\emph{twisted irregular pair}: a certificate that the
$\psi$-component of the class group is nontrivial.

The arithmetic of these twisted divisibilities has a substantial
history.  Ernvall's generalized irregular primes \cite{Ernvall} record
the vanishing of $B_{n,\chi}\bmod p$ for Dirichlet characters; Holden
\cite{Holden} studies the analogous quadratic irregularity and its
distribution.  Most directly, a zero $b_{\chi,j}=0$ is exactly the
statement that the branch $\chi\omega^{-j}$ is irregular, equivalently
that its cyclotomic Iwasawa $\lambda$-invariant is positive.  The
exact order $v_p(B_{1,\chi\omega^{-j}})$ recorded here is a related
but distinct refinement: by \eqref{eq:MC} it is the exact length of
the finite-level isotypic class component, whereas $\lambda$ counts
zeros of the Iwasawa power series and is in general determined only
with the help of derivative information \cite{Knospe}.  The
distribution of these $\lambda$-invariants across characters, and its
dependence on how $p$ splits in $\Q(\chi)$, is the subject of
Delbourgo--Knospe \cite{DelbourgoKnospe}, who tabulate every character
of conductor $\le1000$ for small primes and model the distribution by
$p$-adic random matrix heuristics, and of Knospe \cite{Knospe}, who
gives computable criteria separating $\lambda=0,1,2,\ge3$ and tests the
predicted frequencies.  The present paper is complementary to that
line in three respects, made precise below: it varies $p$ far further
at \emph{fixed} small conductor rather than varying the character at
small $p$; it records \emph{exact} orders of vanishing (second-order
digits) for every zero; and---its main point---it converts the
zeros of the two companion families into \emph{explicit, individually
certified Kummer generators} of the class-field components, which the
invariant computations above do not produce.

The companion papers carry this further.  At conductor three
\cite{TwistedF3} and conductor five \cite{FourierF5}, a universal
Stirling projector polynomial identifies the local Kummer spectrum of
an explicit circular unit with the spectrum
\eqref{eq:b-def}, and each zero line with $p\equiv1\pmod f$ yields an
explicit projected-unit radical whose Kummer extension is everywhere
unramified, certified nontrivial by a single split-prime Artin
symbol, and---by \eqref{eq:MC} together with an exact second-order
digit---generates the \emph{complete} $\psi$-component.  Both papers
close with the same question: how are the zeros $b_{\chi,j}$
distributed as $f$, $\chi$ and $p$ vary?  Their catalogues, at
$p<500$, contain $12$ and $11$ lines respectively.

The companion papers pose, for their two families, the question of how
the zeros $b_{\chi,j}$ are distributed; this paper answers the
computational part of that question at a scale far larger than their
$p<500$ catalogues, and, more substantially, turns the zeros into a
large explicit catalogue of certified class-field generators.  The
contributions are as follows.

\begin{enumerate}
\item \textbf{A uniform weight formula
      (Lemma~\ref{lem:weights}).}  For arbitrary conductor,
      $b_{\chi,j}$ is a weighted power sum
      $\sum_r w_\chi(r\bmod f)r^{-j}$ with at most $f$ distinct
      weights.  The short finite forms of
      \cite[Lemma~3.3]{TwistedF3} and \cite[Lemma~5.4]{FourierF5}
      are the cases $f=3,5$.
\item \textbf{A finite-Fourier evaluation
      (\S\ref{sec:algorithm}).}  A Bluestein chirp factorization
      \cite{Bluestein} evaluates all $b_{\chi,j}$, $2\le j\le p-3$,
      using one polynomial multiplication, hence in
      $O(M(p))=\widetilde O(p)$ arithmetic operations in $\F_p$ for
      fixed conductor.  This matches the order of the standard
      generating-function/power-series method for generalized Bernoulli
      numbers \cite{BCEMS,Harvey}; it is offered as a direct,
      finite-field-native alternative that is convenient in practice,
      not as an asymptotic improvement.
\item \textbf{The survey (\S\ref{sec:survey}--\ref{sec:stats}).}
      $\nZerosTotal$ zero lines over $55{,}121$ character--prime
      pairs ($f\in\{3,5\}$, $p<10^5$; all odd primitive characters
      of conductor $f\le20$, $p<2\cdot10^4$), with every line
      re-proved from the definition by an independent code path and
      depth-checked exactly.  The observed counts, divided digits,
      conjugate-character comparison, and comparison with classical
      irregularity are tested quantitatively against the stated random
      models, in the fixed-conductor, large-$p$ regime complementary
      to the fixed-$p$, varying-character tables of
      \cite{DelbourgoKnospe}.
\item \textbf{Non-simple twisted zeros
      (Theorem~\ref{thm:deep}).}  Eight lines have
      $v_p(B_{1,\chi\omega^{-j}})\ge2$, including one of depth three.
      Each produces an odd class component of order exactly $p^2$ or
      $p^3$, with the valuation certified by an exactly computed
      nonzero leading digit and the Main Conjecture hypotheses
      verified row by row.  None of the eight deep rows in
      Table~\ref{tab:deep} is present in the tabulations of
      \cite{DelbourgoKnospe,Knospe}.  This is a comparison with the
      printed tables, not a priority claim: Appendix~A of
      \cite{DelbourgoKnospe} lists individual characters only for
      $p=3,5,7$, while the explicit list in \cite[\S6.2]{Knospe}
      concerns the related but distinct rank-one branch
      $\chi=\theta\omega$.  We therefore make no priority claim for
      the divisibilities themselves; the exact finite-level
      class-component orders appear to be new.
\item \textbf{The extended conductor-three catalogue
      (Theorem~\ref{thm:extended}).}  For $p<10^5$,
      $p\equiv1\pmod6$, there are exactly $2{,}441$ zero lines; all
      are simple, and every projected circular unit
      $\widetilde\mu_{p-j}$ of \cite{TwistedF3} generates the
      complete order-$p$ component of the Hilbert class field, each
      instance certified by a fresh split-prime witness listed in
      the ancillary files.  The largest lies in the
      degree-$199{,}980$ field $\Q(\zeta_{299973})$
      (Theorem~\ref{thm:record}).
\end{enumerate}

Everything is deterministic integer arithmetic; the ancillary
files (\S\ref{sec:repro}) contain the complete data and a program
that reproduces every table, including independent re-verification
of all certificates.

\section{The residue-class weight formula}
\label{sec:weights}

\begin{lemma}[Residue-class weight formula]
\label{lem:weights}
Let $\chi$, $f$, $d$, $p$ be as above, and for
$c\in\Z/f\Z$ put
\[
 w_\chi(c)=\sum_{m=0}^{f-1}m\,\chi\bigl((c+mp)\bmod f\bigr)
 \in\F_p .
\]
Then, for every even $j$ with $2\le j\le p-3$,
\begin{equation}
\label{eq:weights}
 \boxed{\qquad
 b_{\chi,j}
 =\sum_{r=1}^{p-1}w_\chi(r\bmod f)\,r^{-j}
 \quad\text{in }\F_p .
 \qquad}
\end{equation}
The weights depend only on $r\bmod f$, on $p\bmod f$, and on the
embedded character values.
\end{lemma}

\begin{proof}
Group the sum in \eqref{eq:b-def} by the residue $r=a\bmod p$; the
summands are $a=r+mp$, $0\le m<f$, and all share the Teichm\"uller
factor $\widetilde r^{\,-j}$, which reduces to $r^{-j}$ modulo $p$.
Since $p$ is invertible modulo $f$, the map
$m\mapsto(r+mp)\bmod f$ is a bijection of $\Z/f\Z$, so
$\sum_m\chi(r+mp)=\sum_{c\bmod f}\chi(c)=0$ and
\[
 \sum_{m}(r+mp)\,\chi(r+mp)
 =r\cdot0+p\sum_m m\,\chi(r+mp)
 =p\,w_\chi(r\bmod f).
\]
Dividing by $p$ and summing over $r$ gives \eqref{eq:weights}.
\end{proof}

At $f=3$, $p\equiv1\pmod6$, the weights are $(-1,-1,2)$, recovering
$b_j=3\sum_{r\equiv2\,(3)}r^{-j}$ \cite[Lemma~3.3]{TwistedF3}; at
$f=5$, $p\equiv1\pmod{20}$, they recover the vector
$(-3-s,-3-s,2-s,2+4s,2-s)$ of \cite[Lemma~5.4]{FourierF5}.  We
emphasize that Lemma~\ref{lem:weights} requires only $d\mid p-1$, not
$p\equiv1\pmod f$: the survey below therefore also covers, for
example, $f=3$ with $p\equiv5\pmod6$, where the radical construction
of \cite{TwistedF3} is not yet available.

The same grouping applied with Teichm\"uller lifts modulo $p^2$
yields the second-order form.  With
$\widehat\chi$ the lifted character values,
$\widehat w_\chi(c)=\sum_m m\,\widehat\chi((c+mp)\bmod f)$ in
$\Z/p^2\Z$, and
$T_{\chi,j}=\sum_{a<fp}a\,\widehat\chi(a)\widehat a^{-j}$ the full
Teichm\"uller sum,
\begin{equation}
\label{eq:second-order}
 \frac{T_{\chi,j}}p
 =\sum_{r=1}^{p-1}\widehat w_\chi(r\bmod f)\,\omega(r)^{-j}
 \pmod{p^2},
\end{equation}
so one zero line is depth-checked, and its exact divided digit
$\overline{B_{1,\chi\omega^{-j}}/p}$ computed, in $O(p)$ operations
modulo $p^2$.  Higher digits are obtained from the full sum
$T_{\chi,j}$ modulo $p^K$; the survey uses $K=6$ for the non-simple
lines.

\section{The complete spectrum in quasi-linear time}
\label{sec:algorithm}

Fix a primitive root $g$ modulo $p$ and write $r=g^t$,
$0\le t\le p-2$.  With $n=p-1$ and $a_t=w_\chi(g^t\bmod f)$,
formula \eqref{eq:weights} becomes
\[
 b_{\chi,j}=\sum_{t=0}^{n-1}a_t\,g^{-jt},
\]
a discrete Fourier transform of length $n$ over $\F_p$, evaluated at
every index $j$ simultaneously.  Since $n$ is not smooth in general,
we use the Bluestein factorization \cite{Bluestein}
\[
 jt=\binom{j+t}2-\binom j2-\binom t2 ,
\]
which converts the transform into a single acyclic convolution:
with $u_t=a_tg^{\binom t2}$ and $v_s=g^{-\binom s2}$,
\[
 b_{\chi,j}=g^{\binom j2}\sum_{t=0}^{n-1}u_tv_{j+t},
\]
one polynomial multiplication of length ${\approx}3p$ over $\F_p$.
If $M(n)$ denotes the cost of multiplying degree-$n$ polynomials over
$\F_p$, the complete spectrum therefore costs $O(M(p))$ field
operations and $O(p)$ memory per character.  With quasi-linear
polynomial arithmetic (as supplied in the relevant ranges by FLINT
\cite{FLINT}) this is $\widetilde O(p)$; the Bluestein reduction itself
does not require the stronger assertion $M(p)=O(p\log p)$ for every
prime and every implementation.

\begin{remark}[Relation to the series method]
For the trivial character the analogous task---all
$B_k\bmod p$, $k\le p-3$, equivalently the ordinary irregular
indices---has long been computable in $\widetilde O(p)$ time by
power-series inversion of $(e^x-1)/x$ over $\F_p$, the engine of the
irregular-prime industry \cite{BCEMS}; see also \cite{Harvey} for
the computation of individual Bernoulli numbers.  The same route
extends to arbitrary $\chi$: the generating function
$\sum_{a\le f}\chi(a)te^{at}/(e^{ft}-1)$ yields all
$B_{n,\chi}\bmod p$ by one truncated series inversion and
multiplication, again in $\widetilde O(p)$ for fixed $f$, and this is
how existing systems compute generalized Bernoulli numbers.  The
contribution of Lemma~\ref{lem:weights} plus Bluestein is therefore
not asymptotic: it is a direct finite-field evaluation---no
characteristic-zero or $p$-adic series lift, one convolution, the
character entering only through $f$ weights---which we found simple
to implement, easy to validate against the definition, and fast in
practice.  For variable conductor the stated $O(M(p))$ excludes the
$O(f^2)$ (trivial in our range) cost of forming the weights.
\end{remark}

In the implementation used here (Python driving FLINT), the complete
spectrum at $p\approx10^5$ takes about $0.3$ seconds on one core; the
full conductor-three family below ($9{,}590$ primes,
$2.3\cdot10^8$ indices) took $24$ minutes.  Ordinary irregular
indices for all $p<10^5$ were recomputed alongside by series
inversion, for the independence statistics of \S\ref{sec:stats}.

\section{Survey design and validation}
\label{sec:survey}

Three data sets were computed.

\begin{itemize}
\item \textbf{Family $f=3$.}  The character $\chi_{-3}$; all primes
      $5\le p<10^5$, $p\neq3$, in both residue classes modulo $6$:
      $9{,}590$ primes.
\item \textbf{Family $f=5$.}  Both primitive quartic characters (the
      two $\F_p$-embeddings, as in \cite[\S5]{FourierF5}); all
      primes $p\equiv1\pmod4$, $p\neq5$, $p<10^5$: $9{,}564$
      character--prime pairs.
\item \textbf{Multi-conductor sweep.}  Every odd primitive character
      of conductor $f\le20$, $f\notin\{3,5\}$: thirty-nine
      characters over the eleven conductors
      $f\in\{4,7,8,9,11,13,15,16,17,19,20\}$ (there is no odd
      primitive character modulo $12$), of orders
      $d\in\{2,4,6,10,12,16,18\}$; all primes $p<2\cdot10^4$ with
      $d\mid p-1$.
\end{itemize}

The survey of Bernoulli spectra includes one small non-semisimple pair:
the primitive quadratic character of conductor $11$ at $p=5$, for
which $p\mid\varphi(11)$.  Its spectrum remains valid computational
data, but no class-group interpretation via \eqref{eq:MC} is asserted
for that pair.  Every row of Theorem~\ref{thm:deep}, and every row of
the conductor-three catalogue below, satisfies the semisimplicity
hypothesis.

For each pair $(p,\chi)$ every even index $2\le j\le p-3$ was
tested.  Validation layers, all of which passed:

\begin{enumerate}
\item the Bluestein spectrum agrees with the naive $O(p^2)$ sum for
      every even $j$ at $25$ small character--prime pairs;
\item the spectrum agrees with the defining Teichm\"uller sum
      \eqref{eq:b-def} (an independent mod-$p^2$ code path) at
      seeded random indices;
\item the published catalogues are reproduced exactly: the twelve
      lines of \cite[Theorem~4.1]{TwistedF3} and the eleven lines of
      \cite[Theorem~6.1]{FourierF5}, including every second-order
      digit in the $T/p^2$ and $B/p$ columns of both papers;
\item all $23$ published split-prime Artin certificates re-verify,
      with the exact printed values of the products and exponents;
\item every zero line found by the transform was re-proved from the
      definition through \eqref{eq:second-order} (the depth check
      asserts $v_p(T_{\chi,j})\ge2$); a deterministic
      ${\approx}1\%$ subsample (the $275$ lines with
      $p+j\equiv0\bmod97$) was verified a third time through the
      full $fp$-term sum modulo $p^3$, and the non-simple lines
      through exact sums modulo $p^6$;
\item the recomputed ordinary irregular indices reproduce the
      classical irregular pairs, with irregular-prime density
      $39.5\%$ against the heuristic $1-e^{-1/2}\approx39.35\%$
      \cite{BCEMS}.
\end{enumerate}

\section{The extended catalogues}
\label{sec:catalogues}

Table~\ref{tab:summary} summarizes the counts; the complete tables
are ancillary files (\S\ref{sec:repro}).

\begin{table}[ht]
\centering
\begin{tabular}{lrrrr}
\toprule
family & pairs $(p,\chi)$ & zero lines & expected & $z$\\
\midrule
$f=3$, \ $p<10^5$ & 9,590 & 4,812 & $4,792\pm69$ & $+0.29$\\
$f=5$, \ $p<10^5$, $p\equiv1\ (4)$ & 9,564 & 4,654 & $4,780\pm69$ & $-1.82$\\
sweep $f\le20$, \ $p<2\cdot10^4$ & 35,967 & 18,042 & $17,949\pm134$ & $+0.69$\\
\midrule
total & 55,121 & 27,508 & & \\
\bottomrule
\end{tabular}

\caption{Zero lines found, against the Poisson expectation
$\sum\lambda$, $\lambda=\lfloor(p-3)/2\rfloor/p$ per pair, with
$z=(\text{obs}-\text{exp})/\sqrt{\sum\lambda}$.}
\label{tab:summary}
\end{table}

Of the $\nZerosFthree$ conductor-three lines, $2{,}441$ have
$p\equiv1\pmod6$ and $2{,}371$ have $p\equiv5\pmod6$; of the
$\nZerosFfive$ conductor-five lines, $1{,}192$ have
$p\equiv1\pmod{20}$.  Table~\ref{tab:sample} shows the largest
construction-compatible rows of each family in the format of the
companion papers.  Table~\ref{tab:sweep} breaks the sweep down by
conductor.

\begin{table}[ht]
\centering
\small
\begin{tabular}{crrrrr@{\qquad}crrrrr}
\toprule
$f$ & $p$ & $j$ & $k$ & $T/p^2$ & $B/p$ & $f$ & $p$ & $j$ & $k$ & $T/p^2$ & $B/p$\\
\midrule
3 & 99,829 & 17,032 & 82,797 & 9,959 & 36,596 & 5 & 99,401 & 47,890 & 51,511 & 53,846 & 90,290\\
3 & 99,829 & 84,068 & 15,761 & 97,880 & 65,903 & 5 & 99,401 & 49,098 & 50,303 & 92,577 & 78,156\\
3 & 99,859 & 67,348 & 32,511 & 52,888 & 84,202 & 5 & 99,401 & 90,400 & 9,001 & 36,420 & 7,284\\
3 & 99,907 & 34,020 & 65,887 & 59,813 & 53,240 & 5 & 99,661 & 62,458 & 37,203 & 7,953 & 41,455\\
3 & 99,991 & 61,444 & 38,547 & 8,542 & 69,508 & 5 & 99,721 & 91,482 & 8,239 & 90,552 & 77,943\\
3 & 99,991 & 64,590 & 35,401 & 51,326 & 50,439 & 5 & 99,901 & 2,924 & 96,977 & 8,190 & 1,638\\
\bottomrule
\end{tabular}

\caption{The six largest zero lines with $p\equiv1\pmod f$ in each
of the two main families.  $T/p^2$ and $B/p$ are the exact divided
digits (all nonzero: every listed zero is simple).}
\label{tab:sample}
\end{table}

\begin{table}[ht]
\centering
\begin{tabular}{crrrrr}
\toprule
$f$ & characters & orders & pairs & zero lines & expected\\
\midrule
4 & 1 & 2 & 2,260 & 1,111 & 1,127.4\\
7 & 3 & 2,6 & 4,505 & 2,284 & 2,248.2\\
8 & 1 & 2 & 2,260 & 1,149 & 1,127.4\\
9 & 2 & 6 & 2,248 & 1,132 & 1,121.6\\
11 & 5 & 2,10 & 4,507 & 2,165 & 2,249.4\\
13 & 6 & 4,12 & 4,464 & 2,194 & 2,228.2\\
15 & 1 & 2 & 2,259 & 1,101 & 1,127.2\\
16 & 2 & 4 & 2,250 & 1,151 & 1,122.4\\
17 & 8 & 16 & 2,224 & 1,130 & 1,110.8\\
19 & 9 & 2,6,18 & 6,731 & 3,485 & 3,359.3\\
20 & 1 & 2 & 2,259 & 1,140 & 1,127.2\\
\bottomrule
\end{tabular}

\caption{The multi-conductor sweep, $p<2\cdot10^4$, $d\mid p-1$.}
\label{tab:sweep}
\end{table}

\section{Distributional statistics}
\label{sec:stats}

Under the random model each $b_{\chi,j}$ is uniform in $\F_p$, so
the zero count of a pair $(p,\chi)$ is approximately Poisson with
mean $\lambda_p=\lfloor(p-3)/2\rfloor/p\approx\tfrac12$.  In the
split situation tested here this is the specialization of the
random-model predictions used for $\lambda$-invariant distributions
in \cite{DelbourgoKnospe,Knospe} (vanishing probability $p^{-1}$ per
branch); our fixed-conductor, large-$p$ data complement their
fixed-$p$ tabulations across characters of conductor up to $1000$.

\begin{figure}[ht]
\centering
\includegraphics[width=\textwidth]{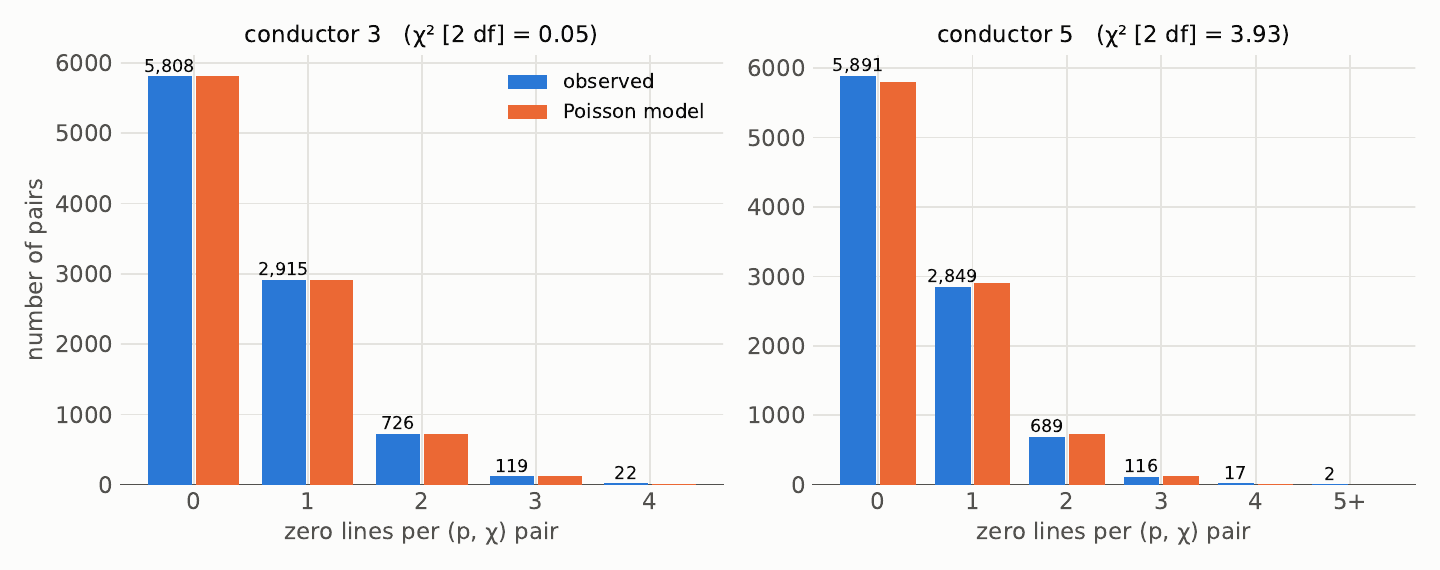}
\caption{Zero lines per pair $(p,\chi)$, observed versus the Poisson
model, for the two main families ($p<10^5$).}
\label{fig:poisson}
\end{figure}

\begin{figure}[ht]
\centering
\includegraphics[width=\textwidth]{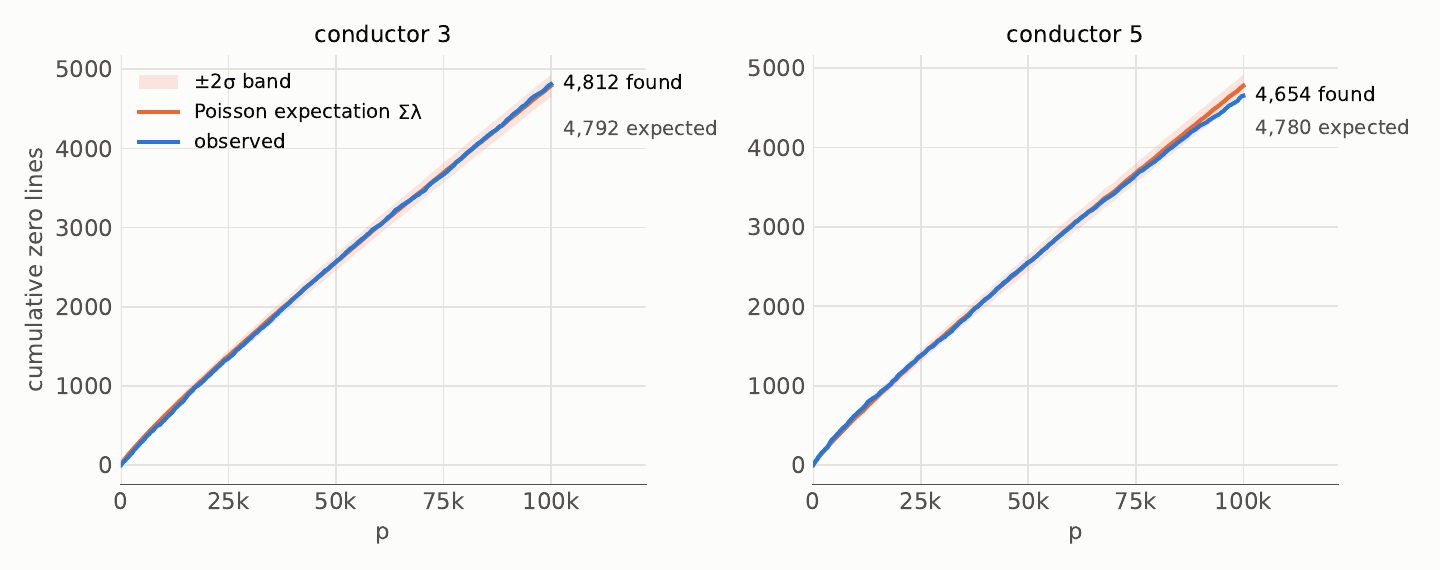}
\caption{Cumulative zero lines against the indexwise random-model
expectation $\sum_{(p,\chi,j)}p^{-1}$, summed over all tested indices
up to the plotted bound, with $\pm2\sigma$ band.}
\label{fig:cumulative}
\end{figure}

\begin{table}[ht]
\centering
\begin{tabular}{crrcrr}
\toprule
$k$ & $f{=}3$ obs. & $f{=}3$ exp. & & $f{=}5$ obs. & $f{=}5$ exp.\\
\midrule
0 & 5,808 & 5,818.4 & & 5,891 & 5,802.2\\
1 & 2,915 & 2,907.4 & & 2,849 & 2,899.7\\
2 & 726 & 726.5 & & 689 & 724.6\\
3 & 119 & 121.0 & & 116 & 120.7\\
$\ge4$ & 22 & 16.8 & & 19 & 16.7\\
\bottomrule
\end{tabular}

\caption{Zero-count histograms against the Poisson expectation.
The $\chi^2$ statistics (2~d.f., tail bucket at $k\ge2$) are
$\chiSqFthree$ for $f=3$ and $\chiSqFfive$ for $f=5$.}
\label{tab:hist}
\end{table}

\textbf{Counts.}  Table~\ref{tab:hist} and
Figures~\ref{fig:poisson}--\ref{fig:cumulative} show the agreement;
every sweep conductor in Table~\ref{tab:sweep} is within
$2.2\sigma$ of its expectation.  The record pair is
$p=36{,}541$, where one quartic character carries six zero lines
($j=4578$, $4968$, $13414$, $26838$, $27938$, $36316$).  At $f=3$
the split by residue class is $+1.0\sigma$ ($p\equiv1$) against
$-0.6\sigma$ ($p\equiv5$): no visible dependence on the splitting
behaviour of $p$ in $\Q(\zeta_f)$.

\textbf{Digits.}  The normalized divided digits
$\overline{B_{1,\chi\omega^{-j}}/p}/p\in(0,1)$ over the $\ksN$
simple zero lines of the two main families pass a
Kolmogorov--Smirnov test against the uniform distribution:
$D=\ksD$ against the $5\%$ critical value $\ksCrit$
(Figure~\ref{fig:digits}).  This is the quantitative content of the
``generic simplicity'' model: within the tested range, the second
digits show no statistically significant departure from uniformity.

\begin{figure}[ht]
\centering
\includegraphics[width=.62\textwidth]{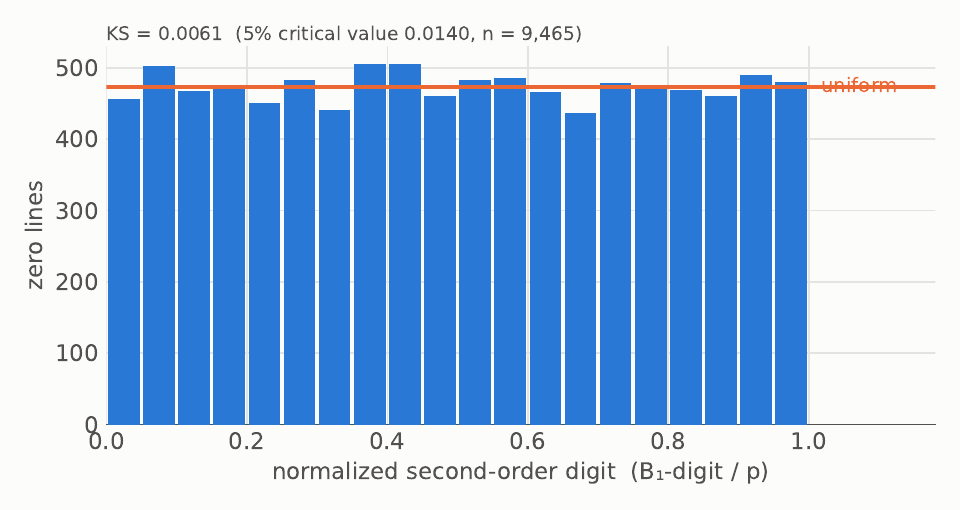}
\caption{Empirical distribution of the normalized second-order digits
over the simple zero lines of the conductor-$3$ and $5$ families
($D=\ksD$, $5\%$ critical value $\ksCrit$, $n=\ksN$).}
\label{fig:digits}
\end{figure}

\textbf{Conjugate-character comparison.}  Among the $\conjPrimes$ primes
carrying both quartic characters modulo $5$, the two spectra vanish
at a common index \emph{zero} times; $\conjBoth$ primes have a zero
for both characters, against $\conjBothExp$ expected under
independence, and the zero-count correlation is $r=\conjCorr$.  The
observation at $p=61$ in \cite[\S10]{FourierF5}---conjugate
characters selecting different indices---is thus the generic
behaviour in this data set, not an accident of the first example.  The
statistics are consistent with the independence model; they do not,
by themselves, prove independence.

\textbf{Comparison with classical irregularity.}  At $f=3$ the
zero rate per pair is $\rateReg$ at ordinarily regular primes and
$\rateIrr$ at irregular primes
($\chi^2(1\text{ d.f.})=\chiSqInd$, not significant), and
$\fracReg\%$ of all conductor-three zero lines occur at classically
\emph{regular} primes.  The ``majority phenomenon'' of
\cite{TwistedF3} persists at scale: ordinary irregularity neither
attracts nor repels twisted degeneracy detectably in the tested range.

\section{Non-simple zeros: components of order \texorpdfstring{$p^2$ and $p^3$}{p2 and p3}}
\label{sec:deep}

Conditioning on the zero lines found by the first-order survey, the
independent-uniform model predicts
\[
 \sum_{\substack{(p,\chi,j)\ \mathrm{in\ the\ survey}\\
                  b_{\chi,j}=0}}\frac1p\approx9.8
\]
non-simple zeros; eight occur.  Before the first-order zeros are
observed, the model has the ex ante indexwise expectation
$\sum_{(p,\chi,j)}p^{-2}\approx11.6$ over all tested indices; after
conditioning on the realized first-order zero set, this becomes the
displayed $\sum p^{-1}$.

\begin{theorem}
\label{thm:deep}
Exactly eight zero lines of the survey have
$v_p(B_{1,\chi\omega^{-j}})\ge2$, listed in
Table~\ref{tab:deep} with their exact valuations and leading divided
digits, computed from the full Teichm\"uller sums modulo $p^6$.
In each row, with $\psi=\overline\chi\,\omega^{j}$,
\[
 \#\bigl(\Cl(\Q(\zeta_{fp}))\otimes\Z_p\bigr)_\psi
 =p^{\,v_p(B_{1,\chi\omega^{-j}})},
\]
as recorded in the table.  In particular the
$\overline\chi\,\omega^{16}$-component at $(f,p)=(19,37)$ has order
$37^3=50{,}653$.
\end{theorem}

\begin{table}[ht]
\centering
\small
\begin{tabular}{ccrrcccr}
\toprule
$f$ & $\mathrm{ord}\,\chi$ & $p$ & $j$ & $v_p(B_{1,\chi\omega^{-j}})$ & leading digit & $\#(\Cl\otimes\Z_p)_\psi$ & $p\bmod f$\\
\midrule
19 & 18 & 37 & 16 & 3 & 9 & $p^3=50,653$ & 18\\
16 & 4 & 73 & 28 & 2 & 40 & $p^2=5,329$ & 9\\
7 & 2 & 173 & 76 & 2 & 127 & $p^2=29,929$ & 5\\
20 & 2 & 193 & 166 & 2 & 155 & $p^2=37,249$ & 13\\
7 & 6 & 211 & 2 & 2 & 15 & $p^2=44,521$ & 1\\
3 & 2 & 257 & 156 & 2 & 137 & $p^2=66,049$ & 2\\
13 & 4 & 373 & 68 & 2 & 25 & $p^2=139,129$ & 9\\
19 & 2 & 2711 & 2670 & 2 & 2559 & $p^2=7,349,521$ & 13\\
\bottomrule
\end{tabular}

\caption{The eight non-simple twisted zeros.  ``Leading digit'' is
$\overline{B_{1,\chi\omega^{-j}}/p^{\,v}}\in\F_p^\times$ with
$v=v_p(B_{1,\chi\omega^{-j}})$; its nonvanishing certifies that the
listed valuation is exact.}
\label{tab:deep}
\end{table}

\begin{proof}
The enumeration is the survey plus validation layer (5) of
\S\ref{sec:survey}; the valuations and digits are exact integer
computations modulo $p^6$, reproduced by the ancillary program.
For the class-group statement, fix a row and put
$\psi=\overline\chi\,\omega^j$.  Then $\psi$ is odd (as $\chi$ is
odd and $j$ even), primitive of conductor $fp$ (as $\chi$ is
primitive modulo $f$ and $j\not\equiv0\bmod p-1$), distinct from
$\omega$, and its values lie in $\Z_p^\times$ since $d\mid p-1$.
Moreover $p\nmid[\Q(\zeta_{fp}):\Q]=\varphi(f)(p-1)$, because
$\varphi(f)\le18<p$ in every row; the isotypic decomposition is
therefore semisimple, and the finite-level odd-character consequence
of the abelian Main Conjecture
\cite[Theorem~4.1]{Greither}, \cite{MazurWiles} gives
\eqref{eq:MC} with
$v_p(B_{1,\psi^{-1}})=v_p(B_{1,\chi\omega^{-j}})$.
A $\Z_p$-module of length $v$ over the discrete valuation ring
$\Z_p$ has order $p^v$.
\end{proof}

\begin{remark}
Three features of Table~\ref{tab:deep} deserve notice.
(i)~The conductor-three double zero occurs at $p=257\equiv5\pmod6$,
where $p$ is inert in $\Q(\zeta_3)$: it lies just outside the
hypotheses of the radical construction of \cite{TwistedF3}, making
the inert-prime extension of that construction a concrete, motivated
problem, with an order-$257^2$ target component.
(ii)~The row $(f,p,j)=(7,211,2)$ is a double zero at the lowest
admissible index.
(iii)~Four of the eight rows occur on quadratic character lines
($\chi_{-3}$, $\chi_{-20}$, and the Legendre characters of
conductors $7$ and $19$), so depth ${\ge}2$ is not a phenomenon of
large character order.
None of the eight rows falls within the $p\equiv1\pmod f$,
$p<500$ ranges of the published catalogues, consistent with the
simplicity findings there.
\end{remark}

\section{The extended conductor-three catalogue}
\label{sec:extended}

Throughout this section $p\equiv1\pmod6$, $K_p=\Q(\zeta_{3p})$, and
$\widetilde\mu_k$ denotes the integral projected circular unit of
\cite[\S2]{TwistedF3} attached to $\theta=\chi_{-3}\omega^{k}$,
$k=p-j$.

\begin{theorem}
\label{thm:extended}
For $p<10^5$, $p\equiv1\pmod6$, the equality $b_{\chi_{-3},j}=0$
holds for exactly $2{,}441$ pairs $(p,j)$.  Every such zero is
simple.  For each, the extension
$K_p(\widetilde\mu_{p-j}^{1/p})/K_p$ is nontrivial, everywhere
unramified, and is the complete
$\psi=\chi_{-3}\omega^{j}$-component of the Hilbert class field of
$K_p$; in particular
\[
 \#\bigl(\Cl(K_p)\otimes\Z_p\bigr)_\psi=p .
\]
\end{theorem}

\begin{proof}
The enumeration and the simplicity of every zero (nonvanishing
divided digit) are the survey computations, doubly verified as in
\S\ref{sec:survey}.  For fixed $(p,j)$ the argument of
\cite{TwistedF3} applies verbatim: $b_{\chi,j}=0$ makes
$\widetilde\mu_{p-j}$ a local $p$-th power at both primes above $p$
\cite[Theorem~6.1, Proposition~6.3]{TwistedF3}, and unramifiedness
away from $p$ is automatic for a unit radical, so the extension is
everywhere unramified.  Nontriviality is certified row by row: the
ancillary file \texttt{certificates\_f3.csv} lists, for each of the
$2{,}441$ pairs, a prime $q\equiv1\pmod{3p}$, an element $\Xi$ of
exact order $3p$ in $\F_q^\times$, the reduced projected product
$M$, and an exponent $e\not\equiv0\pmod p$ with
$M^{(q-1)/p}=\eta_q^{\,e}\neq1$, where $\eta_q=\Xi^3$; a global
$p$-th power would have trivial symbol at every such prime.  The
certificates were generated by one program and re-verified by an
independent implementation.  Completeness then follows from the
simplicity of the zero and \eqref{eq:MC} exactly as in
\cite[Theorem~8.1]{TwistedF3}.
\end{proof}

\begin{remark}
Certificates are not unique: a different split prime $q$, or a
different generator $\Xi$ at the same $q$, gives a different valid
witness.  The ancillary certificates for the twelve rows with
$p<500$ therefore need not coincide with those printed in
\cite{TwistedF3}, although those also re-verify.
\end{remark}

The largest two rows of Theorem~\ref{thm:extended} occur at
$p=99{,}991$, with $j\in\{61444,\,64590\}$.  We record the first in
full; to the author's knowledge it is the largest explicitly generated
class-field component presently recorded in this family.

\begin{theorem}
\label{thm:record}
Let $p=99{,}991$, $K=\Q(\zeta_{299973})$ (a field of degree
$\varphi(299973)=199{,}980$), and
$\psi=\chi_{-3}\omega^{61444}$.  Then
\[
 \#\bigl(\Cl(K)\otimes\Z_p\bigr)_\psi=p=99{,}991,
\]
and $\widetilde\mu_{38547}$ generates the complete
$\psi$-component of the Hilbert class field of $K$.  An explicit
certificate: $b_{\chi,61444}=0$ with divided digits
$T/p^2\equiv8542$ and
$B_{1,\psi^{-1}}/p\equiv69508\not\equiv0\pmod p$, and at
\[
 q=1{,}199{,}893=1+12p,\qquad
 \Xi=81\ (\text{exact order }3p),\qquad
 M=347{,}016,
\]
one has $M^{(q-1)/p}=\eta_q^{\,46540}\neq1$.
\end{theorem}

The certificate was found in $1.2$ seconds and re-verifies in
$0.8$ seconds by schoolbook modular arithmetic; the full
$199{,}980$-term product was recomputed by an independent
implementation.  For comparison, the generic route to any such
statement---\texttt{bnrclassfield} over \texttt{bnfinit}---requires
the unconditional class group and units of the base field.  On the
machine used for the survey, \texttt{bnfinit(polcyclo(3p),1)} in
PARI/GP~2.15.4 \cite{PARI} completed in $0.03$/$0.12$/$0.74$/$50$
seconds at degrees $12$/$24$/$36$/$60$, was killed after $900$
seconds at degree $84$, and at degree $132$---the field
$\Q(\zeta_{201})$ of the first catalogue row of
\cite{TwistedF3}, with $|d_K|\approx10^{269}$---exhausted a
$3.8$\,GB stack ($|d_K|$ for the field of
Theorem~\ref{thm:record} has about $1.05$ million digits).  This is
no criticism of general-purpose tools: it quantifies the gap between
generic and certificate-based access to these components.  Analytic
relative-class-number methods do give per-character
\emph{orders} cheaply; what they do not produce is an explicit
generator with a finite, independently checkable witness.

\section{Concluding remarks}

The survey supplies large-scale evidence for the distributional model
posed in \cite{TwistedF3,FourierF5}.  In the tested range, the counts,
depths, and second digits are consistent with uniform random residues;
no statistically significant dependence is detected either across
conjugate characters or between twisted degeneracy and classical
irregularity.  Four directions are singled out by the data.

\begin{enumerate}
\item \textbf{The inert fork.}  The $2{,}371$ conductor-three lines
      with $p\equiv5\pmod6$---including the depth-two line at
      $(257,156)$---show the Bernoulli side is equally rich where
      the circular-unit construction is not yet available.
      Extending the local analysis of \cite{TwistedF3} to inert
      $p$, with the order-$257^2$ component as first target, is a
      well-posed problem.
\item \textbf{Depth-two radicals.}  Whether the projected unit can
      exhibit a generator of a $p^2$-component---at $(257,156)$, or
      at the eight rows of Table~\ref{tab:deep} generally---is the
      ``depth'' question raised in both companion papers, now with
      concrete instances.
\item \textbf{Range.}  Extending the two main families to $p<10^6$
      costs two to three CPU-days with the present implementation and
      would yield ${\approx}69{,}000$ further lines; but new
      non-simple zeros are \emph{not} expected there
      ($\sum_{(p,\chi,j)}p^{-2}\approx0.2$ over the newly tested
      indices).  Depth hunting is
      better served by widening the conductor sweep, where small
      primes dominate.
\item \textbf{Tables.}  The zero-line tables extend the tradition of
      generalized-irregularity and $\lambda$-invariant tabulation
      \cite{Ernvall,Holden,DelbourgoKnospe,Knospe} in a complementary
      direction---much larger $p$ at fixed small conductor---and add
      two layers those tables do not carry: exact orders of vanishing
      with leading digits, and, on the conductor-three family,
      individual Artin certificates.  The ancillary format is close
      to that of the LMFDB class-group data and is offered for
      incorporation.
\end{enumerate}

\section{Reproducibility and ancillary files}
\label{sec:repro}

All computations are deterministic integer arithmetic (Python
driving FLINT \cite{FLINT}; PARI/GP \cite{PARI} only for the timing
comparison).  The ancillary directory contains:

\begin{itemize}
\item \texttt{anc/zeros\_f3.csv}, \texttt{anc/zeros\_f5.csv},
      \texttt{anc/zeros\_sweep.csv}: all $\nZerosTotal$ zero lines
      with conductor, character label, $p$, $j$, $k=p-j$, exact
      divided digits, simplicity flag, and cross-check flag;
\item \texttt{anc/certificates\_f3.csv}: split-prime certificates
      $(p,j,k,q,\Xi,M,e)$ for all $2{,}441$ rows of
      Theorem~\ref{thm:extended};
\item \texttt{anc/irregular.csv}: ordinary irregular indices for
      all $p<10^5$;
\item \texttt{anc/twisted.py}, \texttt{anc/certificates.py},
      \texttt{anc/verify\_survey.py}: the library and a
      verification program that reproduces the published-catalogue
      check, the validation layers of \S\ref{sec:survey}, the
      valuations of Table~\ref{tab:deep}, and (in full mode) every
      certificate.
\end{itemize}

\section*{Acknowledgements}

Computational exploration, drafting, and verification were assisted
by Anthropic's Claude (Fable~5), which carried out the survey
computations, the statistical analysis, and an adversarial audit of
every numerical claim against the raw data.  Responsibility for the
mathematical statements and final presentation remains with the
author.

\end{document}